\documentclass[12pt]{amsart}
\usepackage[headings]{fullpage}

\usepackage{amsmath,amssymb,amsthm}
\usepackage{graphicx}
\usepackage{enumerate}
\usepackage{url}
\usepackage{slashed}
\usepackage{bm}
\usepackage[german,english]{babel}

\usepackage[bookmarks=true,%
    colorlinks=true,%
    linkcolor=blue,%
    citecolor=blue,%
    filecolor=blue,%
    menucolor=blue,%
    urlcolor=blue,%
    breaklinks=true]{hyperref}

\usepackage[all]{xy}
\usepackage{verbatim}
\usepackage{tikz-cd}

\newtheorem{thm}{Theorem}[]
\newtheorem*{thm*}{Theorem}
\newtheorem{lem}[thm]{Lemma}

\newtheorem{ex}[thm]{Example}
\newtheorem{rem}[thm]{Remark}

\newcommand{\param}{{\mathchoice{\mkern1mu\mbox{\raise2.2pt\hbox{$
\centerdot$}}
\mkern1mu}{\mkern1mu\mbox{\raise2.2pt\hbox{$\centerdot$}}\mkern1mu}{
\mkern1.5mu\centerdot\mkern1.5mu}{\mkern1.5mu\centerdot\mkern1.5mu}}}

\renewcommand{\setminus}{{\smallsetminus}}

\begin{document}

\title{Fibering rigidity of 3-manifolds with Torelli monodromy}
\author{Ingrid Irmer}
\address{Department of Mathematics and Statistics\\
               University of Melbourne\\
	    Parkville, VIC, 3010 \\
	    Australia}       
\email{ingrid.irmer@unimelb.edu.au}

\begin{abstract}
In this paper it is proven that there is at most one way, up to isotopy, in which a connected, hyperbolic, 
orientable 3-manifold can fiber over the circle with monodromy in the Torelli group.
\end{abstract}

\maketitle
\tableofcontents


\section{Introduction}



In his seminal paper \cite{Thurstonnorm}, Thurston described the ways in which a 3-manifold fibers over the circle. These fiberings were parameterised in terms of the rational points of some facets of a unit norm ball in the first rational cohomology of the 3-manifold.\\

When a 3-manifold $M$ fibers over the circle, and its first Betti number $b_{1}(M)$ satisfies $b_{1}(M)\geq 2$, then $M$ fibers in infinitely many ways. The genus of the fiber surfaces can be arbitrarily large, but are always bounded from below by $\frac{b_{1}(M)-1}{2}$. In this paper we therefore study fiberings of 3-manifolds whose fiber surfaces have genus equal to this lower bound. This happens exactly when the monodromy is in the Torelli group, $\mathcal{T}(S)$, i.e. the subgroup of the mapping class group of $S$ that acts trivially on the homology of $S$. Suppose also that the monodromy is pseudo-Anosov, then the genus of a fiber surface of $M$ is at least 2, hence $b_{1}(M)\geq 5$. Despite the fact that such 3-manifolds fiber in infinitely many ways, we obtain the following somewhat unexpected rigidity result, conjectured by Tom Church and Benson Farb, and communicated to the author by Eriko Hironaka. 



\begin{thm}
\label{maintheorem}
Suppose that a closed, connected, orientable 3-manifold fibers over the circle with pseudo-Anosov monodromy in the Torelli group. Of all the infinite ways $M$ can fiber over the circle, the fiber surface with Torelli monodromy is unique up to isotopy.
\end{thm}

\begin{rem}
Monodromies corresponding to isotopic fiber surfaces are conjugate in the mapping class group. Theorem \ref{maintheorem} implies that a pseudo-Anosov monodromy in the Torelli group is unique up to conjugation in the mapping class group.
\end{rem}

\begin{rem}
The pseudo-Anosov assumption in Theorem \ref{maintheorem} is necessary; see Example \ref{example} and Remark \ref{rem.nopA}.
\end{rem}

Fibering rigidity of 4-manifolds, i.e. surface bundles over surfaces, was studied in \cite{Salter}. In that case, one uses the Johnson Kernel instead of the Torelli group, and the proof of a somewhat different result makes use of the Johnson homomorphism theory. A survey of ``fibering rigidity'' results for 4-manifolds is given in \cite{Rivin}.\\

\textbf{Plan of proof.} Theorem \ref{maintheorem} is proven by contradiction. Suppose there are two fibrations of a hyperbolic manifold $M$, with nonisotopic fibers $S_1$ and $S_2$ and monodromies $\tau_{1}\in \mathcal{T}(S_{1})$ and $\tau_{2}\in \mathcal{T}(S_{2})$. A covering space of $M$ is found, that retracts onto a common infinite cyclic covering space $S_3$ of $S_1$ and $S_2$. The assumption that the monodromies are in the Torelli subgroup is used to show that
a certain lift or lifts of the monodromies act trivially on homology. This contradicts the fact that the image of $H_{1}(S_{3};\mathbb{Z})$ in $H_{1}(M;\mathbb{Z})$ has infinite rank.\\

\textbf{Outline of the paper.} Subsection \ref{notation} introduces some basic notation, conventions and definitions that will be used throughout the paper. Subsection \ref{intersection} deals with maps between the homology of covering spaces, and the global structure of such covering spaces. In Subsection \ref{translation}, some consequences of hyperbolicity are established. Finally Section \ref{final} uses homological arguments to show that there is a product of powers of the monodromies $\tau_{1}$ and $\tau_{2}$ whose lift corresponds to a deck transformation that acts trivially on the homology of a covering space of $M$. In Section \ref{lastone} this is shown to give a contradiction from which Theorem \ref{maintheorem} follows.\\


\subsection*{Acknowledgements} Thanks to E. Hironaka for suggesting this problem to me, to J. Birman and B. Farb for helpful discussions of background work, and to C. Leininger for suggesting the example in Figure \ref{LMfigure}. The author also wishes to thank S. Friedl, A. Hatcher and N. Salter for their comments on an earlier version of this paper; particularly to A. Hatcher and N. Salter for devoting considerable time to making detailed comments and improvements. Figure \ref{nicksfigure} was pointed out to the author by N. Salter.


\section{Background and basic structure}

This Section introduces the formalism and assumptions needed in the proof of Theorem \ref{maintheorem}. Subsection \ref{intersection} derives basic properties of the covering spaces used, and provides a means of visualising their global structure. This is followed by an example in Subsection \ref{translation} to illustrate the consequences of the assumption that the monodromies of the fiberings are pseudo-Anosov.\\

\subsection{Conventions and notations}\label{notation}
Throughout this paper, $M$ will denote a closed, connected, oriented 3-manifold. A fiber in $M$ determines a free homotopy class of surfaces in $M$, as well as an element of $H^{1}(M, \mathbb{Z})$ which is the pullback of the generator of the first cohomology of the base space over which $M$ fibers. When it does not cause confusion, the same notation will be used for a surface embedded in $M$ representing the homotopy class and the fiber. \\

Unless otherwise stated, by \textit{curve} is meant here a nontrivial free homotopy class of maps from the 1-sphere into a surface or 3-manifold. Curves will sometimes be assumed to pass through a basepoint. This basepoint in $M$ is chosen to be in the intersection of the pair of embedded surfaces $S_1$ and $S_2$ representing the two fibers. A curve will often be confused with the image in a surface or 3-manifold of a chosen representative of the homotopy class.  All curves, surfaces and 3-manifolds are assumed to be oriented. \\



Let $\tau_1$ and $\tau_{2}$ denote the monodromies of the fiberings with fibers $S_1$ and $S_2$ respectively. Let $M_{1}$ and $M_{2}$ be the infinite cyclic coverings $M_{1}\simeq S_{1}\times \mathbb{R}$ and $M_{2}\simeq S_{2}\times \mathbb{R}$. Denote by $M_{3}$ the smallest covering space of both $M_{1}$ and $M_{2}$; in other words, $M_{3}$ is the quotient of the universal cover $\tilde{M}$ of $M$ modulo the subgroup $\pi_{1}(M_{1}) \cap \pi_{1}(M_{2})\subset \pi_{1}(M)$. \\

It follows from \cite{Thurstonnorm}, Section 3, that nonisotopic fibers always intersect. Since we are assuming $S_{1}$ and $S_{2}$ are not isotopic, they intersect, hence $\pi_{1}(M_{3})$ is nontrivial.\\

\begin{rem}
The reader should be warned that in what follows, we will be very relaxed about where a curve lives. Any two conjugacy classes of curves $c_1$ and $c_2$ in $M$ that lift to closed curves in $M_{3}$ can be homotoped onto both of the fibers $S_1$ and $S_2$, and the same notation will be used for these curves in the fibers.  It will always be stated in what space we are working, and all curves mentioned should be assumed to be in that space.
\end{rem}

\subsection{The covering space $M_{3}$}\label{intersection}



The 3-manifold $M_{3}$ is a covering space of manifolds $M_{1}$ and $M_{2}$, each of which retract onto a surface, hence $M_{3}$ also retracts onto a surface. The surface onto which $M_{3}$ retracts will be denoted by $S_{3}$. We will show later (see Lemma \ref{infinitecyclic} below) that $S_{3}$ is a common infinite cyclic cover of $S_1$ and $S_2$. A subtle observation due to A. Hatcher is that we will be considering two product structures on $M_{3}$; these will be denoted by $S_{3}\times_{1} \mathbb{R}$ and $S_{3}\times_{2}\mathbb{R}$. The pre-images of the fibers in $M_{3}$ can both be shown to be homeomorphic, but it is not clear that there is a canonical homeomorphism. \\

Denote by $S_{3,1}$ a representative of the homotopy class of $S_{3}$ in $M_{3}$ consisting of a connected component of the pre-image of $S_{1}$, and likewise by $S_{3,2}$ a representative of the homotopy class of $S_{3}$ in $M_{3}$ consisting of a connected component of the pre-image of $S_{2}$. A set of curves on $S_{3,1}$ representing a basis for $H_{1}(S_{3,1};\mathbb{Z})\simeq H_{1}(M_{3};\mathbb{Z})$ can be homotoped onto a set of curves on $S_{3,2}$. This latter set of curves must then also represent a basis for $H_{1}(S_{3,2};\mathbb{Z}) \simeq H_{1}(M_{3};\mathbb{Z})$. It follows that the two different product structures on $M_{3}$ give the same notion of algebraic intersection number, in the sense that homotoping two curves in $M_{3}$ onto $S_{3,1}$ and computing algebraic intersection number in $S_{3,1}$ will give the same result as homotoping the curves onto $S_{3,2}$ and computing algebraic intersection number in $S_{3,2}$.\\


\textbf{Maps between homology.} Figure \ref{commutative} shows the induced maps on homology from the covering spaces.

\begin{figure}
\centering
\def\svgwidth{7cm}
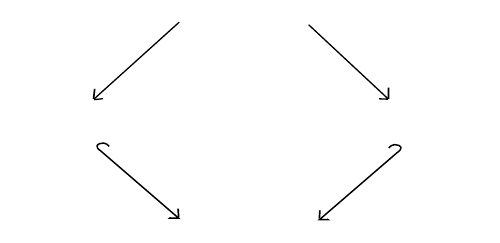
\caption{}
\label{commutative}
\end{figure}

\begin{lem}
\label{lemmachen}
The diagram in Figure \ref{commutative} commutes, and the maps from $H_{1}(M_{1};\mathbb{Z})$ and $H_{1}(M_{2};\mathbb{Z})$ 
into $H_{1}(M;\mathbb{Z})$ are injective.
\end{lem}

\begin{proof}
Commutativity follows from the fact that the diagram shows the maps on homology induced by a commutative diagram of covering spaces.\\

The injectivity claim uses the fact that the monodromies are in the Torelli groups of the surfaces representing the fibers. It is a consequence of this assumption that the embeddings of the fibers in $M$ induce isomorphisms of $H_{1}(M_{1};\mathbb{Z})$ and $H_{1}(M_{2};\mathbb{Z})$ with their images in $H_{1}(M; \mathbb{Z})$.
\end{proof}

We now want to be able to visualise $M_{3}$. The rest of this section derives general properties of $M_{3}$ that, unless otherwise stated, do not depend on the assumptions that the monodromies are pseudo-Anosov elements of the Torelli group. From now on it will be assumed that the two surfaces representing the fibers above any pair of points in the pair of base spaces are in general and minimal position. In particular, $S_{1}$ and $S_{2}$ are embedded, intersect transversely and minimally.\\

\begin{figure}[h]
\centering
\def\svgwidth{7cm}
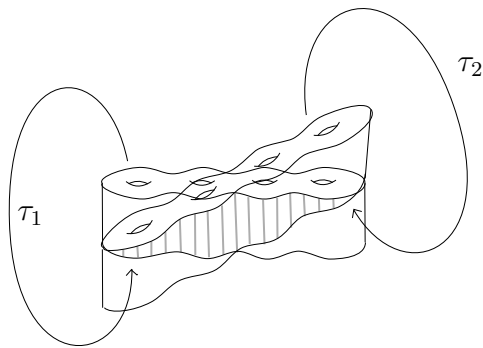
\caption{A ``diamond'', shown in grey lines.}
\label{diamonds}
\end{figure}

\textbf{Diamonds.} Fix a Riemannian metric on $M$. It induces Riemannian metrics on $M_{3}$, $M_{1}$, $M_{2}$ and $M\setminus (S_{1}\cup S_{2}).$ The interior of a diamond is a connected component of $M\setminus (S_{1}\cup S_{2})$ as shown in Figure \ref{diamonds}. A diamond is the metric space completion of its interior. Diamonds lift isometrically to $M_{3}$.To see why, note that there is a deformation retract of a diamond $d$ onto a connected component of the complement in $M$ of a regular neighbourhood of $S_{1}\cup S_{2}$. Let $M\rightarrow S^{1}$ denote the fibering with fiber $S_1$. Then the composition $d\hookrightarrow M \rightarrow S^{1}$ induces the trivial map on $\pi_{1}(M)$. Similarly for the fibering with fiber $S_2$. It follows that the embedding $d\hookrightarrow M$ lifts to the covering $M_{3}$ corresponding to the intersection of the kernels of the two maps.\\

A diamond has two projection maps onto $S_{3}$ via each of the product structures $\times_{1}$ and $\times _{2}$. In this sense, a diamond can be collapsed onto the union of a subsurface of $S_{1}$ and a subsurface of $S_{2}$.\\


\textbf{Example of a diamond.} The following example of a pair of fibers in a 3-manifold fibering over the circle is taken from \cite{LM}, Lemma 5.1. \\

\begin{ex}
\label{diamondexample}
For a simple curve $\alpha$ in a surface, a Dehn-twist around $\alpha$ is denoted by $T_{\alpha}$. Let $M$ be the 3-manifold 
fibering over the circle with fiber $S_1$ and monodromy given by the pseudo-Anosov 
\begin{equation*}
\tau_{1}:=T^{-1}_{\beta_{0}}T_{\alpha_{0}}T_{\alpha_{1}}T^{-1}_{\beta_{1}}
\end{equation*}
where the curves $\alpha_{i}$ and $\beta_{i}$ for $i=0,1$ are shown in Figure \ref{LMfigure}. It is assumed that the Dehn twists are applied from right to left. The curve $\gamma$ shown in Figure \ref{LMfigure} has the property that $\tau_{1}(\gamma)-\gamma$ bounds an embedded surface with genus one in $S_{1}$. When $M$ is cut along the fiber $S_1$, there is therefore a genus two surface $\Sigma$  in $M$ whose intersection with $M\setminus S_{1}\simeq S_{1}\times I$ is a surface with boundary $\tau_{1}(\gamma)-\gamma$.\\

It is argued in \cite{LM}, Lemma 5.1, that the homology class $[S_{1}]+[\Sigma]$ is in the interior of a cone over a fibered face. The surface $S_2$ with $[S_{2}]=[S_{1}]+[\Sigma]$, whose intersection with $S_{1}\times I$ is illustrated in Figure \ref{s2}, is therefore also a fiber, with pseudo-Anosov monodromy. The monodromy of the fiber $S_2$ can not be in the Torelli group, because the genus of $S_2$ is larger than the genus of $S_1$. 
\end{ex}

As shown in Figure \ref{s2}, when $S_{1}\times I$ is cut along its intersection with $S_2$, a diamond is obtained. The curve $g_1$ corresponds to a generator of the deck transformation group of the infinite cyclic cover $S_{3}\rightarrow S_{1}$, and the curve $g_2$ corresponds to a generator of thedeck transformation group of the infinite cyclic cover $S_{3}\rightarrow S_{2}$. Figure \ref{s2} is slightly misleading; $(S_{1}\times I)\setminus S_{2}$ is actually connected. This diamond has two copies of $S_{2}\setminus S_{1}$ on its boundary, and two copies of $S_{1}\setminus S_{2}$ on its boundary. The cover $M_{3}$ is obtained by stacking diamonds on top of each other and next to each other, in such a way that the fibers match up. In $S_{3}\times_{2}\mathbb{R}$ the cover $S_3$ of $S_2$ is obtained by attaching many copies of $S_{2}\cap (S_{1}\times I)$. The cover $S_3$ of $S_1$ is obtained by cutting $S_1$ along the curve in the intersection of $S_1$ with $S_2$ and attaching the pieces to form an infinite cyclic cover, as shown in Figure \ref{covers1}. These two covering spaces give two different embeddings of $S_3$ in $M_3$.\qed

\begin{figure}
\centering
\def\svgwidth{12cm}
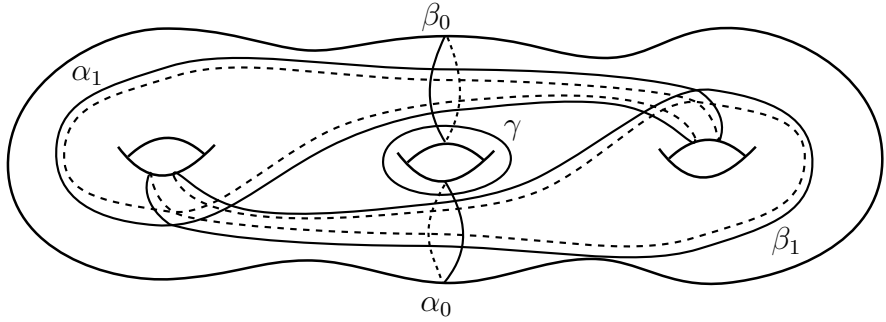
\caption{The curves $\alpha_0$, $\alpha_{1}$, $\beta_{0}$ and $\beta_{1}$ in the definition of $\tau_{1}$ from the example of a diamond. Note that $\tau_{1}(\gamma)-\gamma$ bounds a genus one subsurface of $S_1$.}
\label{LMfigure}
\end{figure}

\begin{figure}
\centering
\def\svgwidth{11cm}
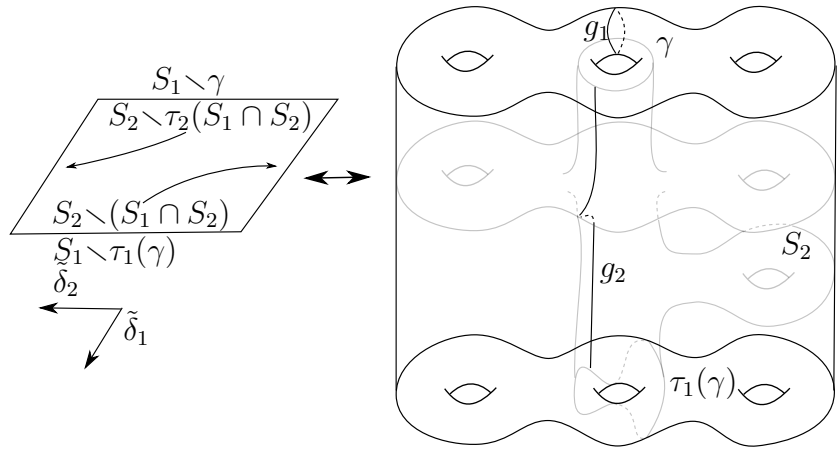
\caption{The intersection of $S_2$ with $S_{1}\times I$ in the example of a diamond is shown in grey. The arrows indicate how the deck transformations $\delta_{1}$ and $\delta_{2}$ translate the diamond.}
\label{s2}
\end{figure}

\begin{figure}
\centering
\includegraphics[width=5cm]{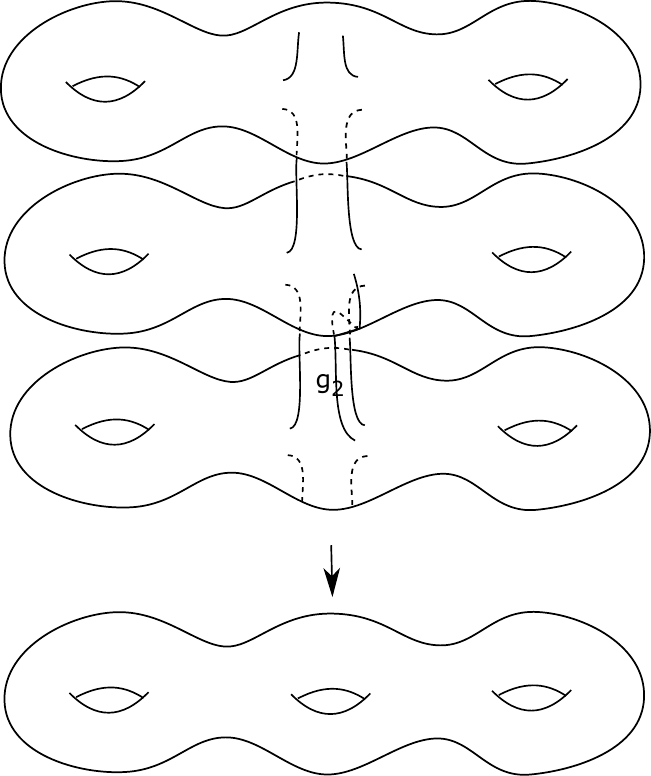}
\caption{The infinite cyclic cover $S_3$ of $S_1$. The lift of the curve $g_2$ from $M$ is shown.}
\label{covers1}
\end{figure}

\pagebreak

\begin{lem}
\label{infinitecyclic}
The covers $S_3$ of $S_1$ and $S_3$ of $S_2$ are infinite cyclic. Moreover, the deck transformation group of the cover $M_{3}\rightarrow M$ is Abelian.
\end{lem}

\begin{proof}
Let $\phi_{1}$ and $\phi_{2}$ be the surjections of $\pi_{1}(M)$ onto $\mathbb{Z}$ defining the fiberings with fibers $S_1$ and $S_2$ respectively. Alternatively, $\phi_1$ and $\phi_2$ can be thought of as defining the covering spaces $M_{1}$ and $M_{2}$ of $M$. Recall that the fundamental group of the cover $M_{3}$ of $M$ is the smallest covering space of both $M_{1}$ and $M_{2}$. The fundamental group of $M_{3}$ is therefore given by the intersection of the kernels of $\phi_{1}$ and $\phi_{2}$. This is isomorphic to the kernel of the product homomorphism $\phi_{1}\times \phi_{2}: \pi_{1}(M)\rightarrow \mathbb{Z}\times \mathbb{Z}$. The deck transformation group of the cover $M_{3}$ of $M$ is therefore an infinite subgroup of $\mathbb{Z}\times \mathbb{Z}$, from which it follows that it must be isomorphic to an infinite cyclic group or a free Abelian group of rank two. \\

\begin{figure}
\[
\xymatrix{
		& 1 \ar[d]						& 1 \ar[d]						& 0 \ar[d]			& \\
1 \ar[r]	&\pi_{1}(S_{3}) \ar[r] \ar[d]			& \pi_{1}(S_{2}) \ar[r]^{\phi_{1}} \ar[d] 		& k\mathbb{Z} \ar[d]	\ar[r]		&0\\
1 \ar[r]	& \pi_{1}(S_{1}) \ar[r] \ar[d]_{\phi_{2}}		& \pi_{1}(M) \ar[r]^{\phi_{1}} \ar[d]_{\phi_{2}}	& \mathbb{Z} \ar[d] \ar[r]		& 0\\
0 \ar[r]	&k\mathbb{Z} \ar[r] \ar[d]					& \mathbb{Z} \ar[r]	\ar[d]					& \mathbb{Z}/k\mathbb{Z}\ar[r]  	\ar[d]		& 0\\
		& 0							& 0							&0				&
}
\]
\caption{A commutative diagram of the covering spaces with all rows and columns exact, from \cite{Rivin}.}
\label{nicksfigure}
\end{figure}

In Figure \ref{nicksfigure} exactness of the middle row and column are given. That the remaining rows and columns are exact is a consequence of the third isomorphism theorem. Since $S_1$ is not isotopic to $S_2$ the integer $k$ can not be zero. The claim that $S_3$ is an infinite cyclic cover of $S_1$ and $S_2$ can be seen by reading off the first column and row respectively.
\end{proof}

\textbf{Lifts of monodromies.} The monodromy $\tau_1$ of the fibration of $M$ corresponding to $S_{1}$ determines a deck transformation $\delta_{1}$ of the cover $M_{1}\simeq S_{1}\times \mathbb{R}$ of $M$ of the form $(x,t)\mapsto (\tau_{1}(x), t+1)$. We claim that $\delta_{1}$ can be lifted to a homeomorphism $\tilde{\delta}_{1}:M_{3}\rightarrow M_{3}$ of the form $(x,t)\mapsto (\tilde{\tau}_{1}(x), t+1)$, where $\tilde{\tau}_{1}:S_{3,1}\rightarrow S_{3,1}$. This is because the action of $\delta_{1}$ on $\pi_{1}(M)$ takes the normal subgroup $\pi_{1}(M_{3})$ to itself. Recall that $\pi_{1}(M_{3})$ is isomorphic to $\pi_{1}(S_{1})\cap \pi_{2}(S_{2})=Ker(\phi_{1})\cap Ker(\phi_{2})$, and is a subgroup of $\pi_{1}(M)$. In $\pi_{1}(M)$, we saw in Lemma \ref{infinitecyclic} that the action of $\delta_{1}$ is conjugation by a loop dual to the surface $S_1$. Since each of $Ker(\phi_{1})$ and $Ker(\phi_{2})$ are normal subgroups of $\pi_{1}(M)$, they are each mapped to themselves by conjugation, and hence so is their intersection. A symmetric argument shows that $\delta_{2}$ can also be lifted. \\

The lift to $S_{3,1}$ of $\tau_{1}$, $\tilde{\tau}_{1}$, is only defined up to deck transformation of the cover $S_{3,1}\subset S_{3}\times_{1} \mathbb{R}$ of $S_{1}$, and the lift $\tilde{\tau}_{2}$ of $\tau_{2}$ is only defined up to a deck transformation of the cover $S_{3,2}\subset S_{3}\times_{2} \mathbb{R}$ of $S_{2}$.

The next lemma collects several elementary statements.

\begin{lem}
\label{andyslemma}
\begin{enumerate}
\item The covering space $M_{3}$ is tiled by diamonds.
\item In $M_{3}$, the lifts $\tilde{\delta}_1$ and $\tilde{\delta}_2$ map diamonds to diamonds. 
\item There are fundamental domains for the covering $M_{3}$ of $M$ consisting of a finite union of diamonds.
\item There are lifts $\bar{\delta}_{1}$ and $\bar{\delta}_{2}$ of $\delta_{1}$ and $\delta_{2}$ respectively such that the deck transformation group of the cover $M_{3}\rightarrow M$ is generated by a pair of elements corresponding to words of the form $\bar{\delta}^{k_{1}}_{1}\bar{\delta}^{k_{2}}_{2}$, and $\bar{\delta}^{k_{3}}_{1}\bar{\delta}^{k_{4}}_{2}$, $k_{1}, k_{2}, k_{3}, k_{4}\in \mathbb{Z}$. 
\end{enumerate}
\end{lem}

Note that $\tilde{\delta}_{1}$ and $\tilde{\delta}_{2}$ commute, by Lemma \ref{infinitecyclic}.

\begin{proof}
To show (1), note that diamonds lift homeomorphically to $M_{3}$, and that the pre-images of $S_1$ and $S_2$ cut $M_{3}$ into diamonds.\\

The statement (2) is a consequence of (1) and the fact that each of $\tilde{\delta}_1$ and $\tilde{\delta}_2$ maps the pre-image of $S_1$ onto itself, and the pre-image of $S_2$ onto itself.\\



Recall that $S_1$ and $S_2$ are assumed to be in minimal position, so there are only finitely many connected components of $S_{1}\cap (M\setminus S_{2})$. Part (3) stating that there can be only finitely many diamonds in a fundamental domain for the cover $M_{3}$ of $M$ is a consequence of this. Figure \ref{s2} shows a pair of fibrations for which there is only one diamond in the fundamental domain. However, the multicurve $S_{1}\cap S_{2}$ in $S_1$ could cut $S_1$ into more than one connected component, in which case there will be more than one diamond. \\


The statement (4) is a corollary of (2) and (3).
\end{proof}

\textbf{Interpreting Figures} A schematic representation of the action of $\tilde{\tau}_{1}$ on $M_{3}$ is shown in Figure \ref{decktransformationgroup}. In this and later figures, horizontal lines will represent the connected components of the pre-image of $S_1$, i.e. copies of $S_{3}$, and slanted lines will represent connected components of the pre-image of $S_2$. In all figures the vertical direction represents the direction in which the $\mathbb{R}$ coordinate varies in the product structure $\times_{1}$. Dots or line segments placed vertically above each other will represent freely homotopic curves or subsurfaces.

\begin{figure}[h]
\centering
\def\svgwidth{10cm}
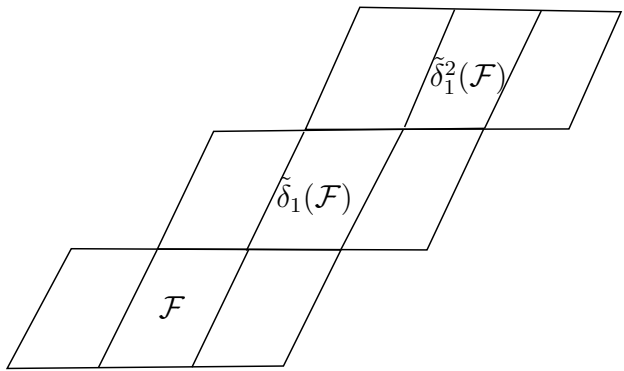
\caption{A fundamental domain $\mathcal{F}$ for the covering $M_{3} \rightarrow M_{1}$, and an example of how $\tilde{\delta}_1$ might translate $\mathcal{F}$.} 
\label{decktransformationgroup}
\end{figure}

\subsection{Monodromies that are not pseudo-Anosov}
\label{translation}

In this subsection we discuss examples of manifolds for which $\tau_1$ and $\tau_2$ are not pseudo-Anosov. The next example illustrates a phenomenon that the pseudo-Anosov assumption on the monodromies rules out. \\

\begin{ex}[Fibrations with Torelli monodromy that are not pseudo-Anosov]
\label{example}
Let $\tau_1$ be a mapping class in the Torelli group that is not pseudo-Anosov and leaves a simple nonseparating curve $s$ on $S_1$ invariant. There is then a torus $T$ in $M$ that intersects $S_1$ along the curve $s$. Suppose $S_2$ is obtained from $S_1$ by Dehn twisting around $T$ in a direction transverse to $S_1$. The fibers can not be isotopic, because they represent elements of $H_{2}(M;\mathbb{Z})$ differing by the homology class $[T]$. Since $s$ is nonseparating, $[T]$ can not be the trivial homology class.\\

Since both fibers have the same genus, both monodromies must be in the Torelli group. \qed
\end{ex}

In Example \ref{example} the fiberings have non-isotopic fibers. However, the Dehn twist defined above represents a homeomorphism of $M$ onto itself, taking $S_1$ to $S_2$. 
 

\section{Actions of the deck transformation groups on homology}
\label{final}

In this section we discuss the action of the deck transformations of the cover $M_{3}\rightarrow M$ on the integral homology of $M_{3}$. When the monodromies $\tau_{1}$ and $\tau_{2}$ are pseudo-Anosov, this leads to a contradiction to the existence of the two nonisotopic fibers $S_{1}$ and $S_{2}$ with Torelli monodromy. \\

\textbf{Intersection curves }The product structure $\times_{1}$ on $M$ gives rise to a map $h_{1}:S_{1}\times [0,1]\rightarrow M$ such that at each $t$, $h_{1}:S_{1}\times\{t\} \rightarrow M$ is an embedding with image $S_{1,t}$. Likewise for $\times_{2}$. We assume that for generic $t_{1}$, $t_{2}\in [0,1]$, the surfaces $S_{1,t_{1}}$ and $S_{2,t_{2}}$ in $M$ intersect transversely. A curve $c$ in $M$ is called an intersection curve if $c\in S_{1,t_{1}}\cap S_{2,t_{2}}$ for some $t_{1}, t_{2} \in [0,1]$.\\



As $t_{2}$ varies in $[0,1]$, $S_{2,t_{2}}$ sweeps out $M$. Moreover, for fixed $t_{1}\in [0,1]$, the intersection curves in $S_{2,t_{2}}\cap S_{1,t_{1}}$ give rise to a singular 1-dimensional foliation of $S_{1, t_{1}}$. This foliation determines a set of intersection curves that span half the homology of $S_{3}$.\\

Recall that by Lemma \ref{infinitecyclic}, the covers $S_{3,1}\rightarrow S_{1}$ and $S_{3,2}\rightarrow S_{2}$ are cyclic. Let $[s]$ be a homology class in $H_{1}(M;\mathbb{Z})$ such that fundamental domains for the covers $S_{3,1}\rightarrow S_{1}$ and $S_{3,2}\rightarrow S_{2}$ respectively are obtained by cutting $S_{1}$ and $S_{2}$ along curves in the homology class $[s]$. \\

Denote by $d_{1}$ a generator of the deck transformation group of the cover $S_{3,1}\rightarrow S_{1}$ and by $d_{2}$ a generator of the deck transformation group of the cover $S_{3,2}\rightarrow S_{2}$. It will be assumed that $d_{1}$ and $d_{2}$ each shift an oriented curve $\tilde{s}$ projecting onto $s$ to its right. The action of $d_{1}$ on curves corresponds to conjugation in $\pi_{1}(S_{3})$ by an arc $\tilde{s}_{1}^{*}$. Here $s_{1}^{*}$ begins at the base point of $S_{3}$ and projects onto a closed curve in $S_{1}$. Similarly the action of $d_{2}$ on curves corresponds to conjugation by an arc $\tilde{s}_{2}^{*}$.\\

The next lemma gives a means of constructing a basis for the homology $H_{1}(S_{3};\mathbb{Z})$.

\begin{lem}
\label{basis}
Let $\{a_{i}\}$ be a maximal, nonseparating set of pairwise disjoint intersection curves on $S_{1}$ such that $[s]$ is not in the span of $\{[a_{i}]\}$ in $H_{1}(S_{1};\mathbb{Z})$. Then the connected components of the pre-images of the set of curves $\{a_{i}\}$, their Poincar\'e duals $\{b_{j}\}$, and one connected component of the pre-image of $s$, represent a basis for $H_{1}(S_{3};\mathbb{Z})$.
\end{lem}

\begin{proof}
This can be easily verified by computation.
\end{proof}

Let $\tilde{c}$ be a connected component of the pre-image in $M_{3}$ of a nonseparating curve $c$ in $M$, $[c]\neq[s]$. Lemma \ref{basis} will be used to show that, for example, $\tilde{c}$ can not be homologous in $S_{3}$ to its conjugate by $\tilde{s}_{1}^{*}$ or $\tilde{s}_{2}^{*}$.\\

Note that deck transformations take null homologous curves to null homologous curves. It follows that the deck transformations of the cover $M_{3}\rightarrow M$ induce actions on the homology $H_{1}(M_{3};\mathbb{Z})$. The actions of the deck transformations $d_{1}$ and $d_{2}$ on $H_{1}(M_{3};\mathbb{Z})$ will be denoted by $d_{1}^{*}$ and $d_{2}^{*}$ respectively. The next lemma describes how deck transformations act on homology.\\

\begin{lem}
\label{eitheror}
At least one of the deck transformation $d_{2}^{-1}d_{1}$ and $d_{1}d_{2}$ acts trivially on $H_{1}(M_{3},\mathbb{Z})$. 
\end{lem}

\begin{proof}
We will prove the lemma using a basis for $H_{1}(M_{3};\mathbb{Z})$ given by Lemma \ref{basis}, and a case analysis of the position of an intersection curve $\tilde{a}_{i}$ or its Poincar\'e dual $\tilde{b}_{j}$.\\

Denote by $M^{k}$ the $k$-fold cyclic cover of $M$, where $k$ is the integer from the commutative diagram of Figure \ref{nicksfigure}. Let $\mathcal{D}$ be a fundamental domain of the cover $M_{3}\rightarrow M^{k}$. The fundamental domain $\mathcal{D}$ consists of $k^{2}$ diamonds, as shown in Figure \ref{mathcald}. These diamonds are stacked beside and on top of each other in such a way that $\partial \mathcal{D}$ is a union of four surfaces with boundary. Two of the surfaces, $\mathcal{F}_{1}$ and $d_{2}(\mathcal{F}_{1})$ are in the pre-image of $S_{1}$, and the other two surfaces, $\mathcal{F}_{2}$ and $d_{1}(\mathcal{F}_{2})$ are in the pre-image of $S_{2}$. By construction, $\mathcal{F}_{1}$ is a fundamental domain of the cover $S_{3,1}\rightarrow S_{1}$, and $\mathcal{F}_{2}$ is a fundamental domain of the cover $S_{3,2}\rightarrow S_{1}$.\\


\begin{figure}
\centering
\def\svgwidth{8cm}
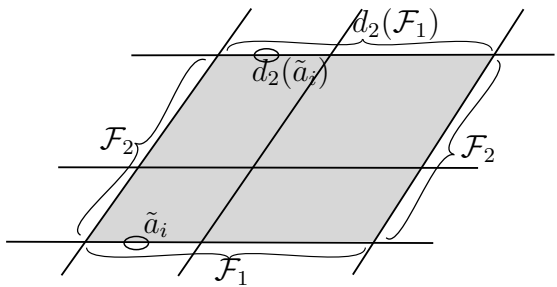
\caption{A schematic representation of $\mathcal{D}$ with $k=2$ is shown. The horizontal lines represent connected components of the pre-image of $S_{1}$, and the slanted lines represent connected components of $S_{2}$.}
\label{mathcald}
\end{figure}

To begin with, fix an intersection curve $a_{i}$ as in Lemma \ref{basis}. Denote by $\tilde{a}_{i}$ a connected component of the pre-image in $S_{3,1}\subset M_{3}$ of $a_{i}$. It is assumed that $\tilde{a}_{i}$ is contained in the surface $\mathcal{F}_{1}\subset \partial \mathcal{D}$, and hence $d_{2}(\tilde{a}_{i})$ is also on $\partial \mathcal{D}$ by construction.\\

Since $\tau_{1}\in \mathcal{T}(S_{1})$ and $\tau_{2}\in \mathcal{T}(S_{2})$, by Lemmas 3 and \ref{andyslemma} part (4), the projection to $S_{1}$ of $\tilde{a}_{i}$ is homologous in $M$ to the projection of $d_{2}(\tilde{a}_{i})$. Consequently there is an embedded surface $\Sigma$ in $M$ with boundary the projection to $M$ of $\tilde{a}_{i}-d_{2}(\tilde{a}_{i})$.\\


We claim that no lift of $\Sigma$ has interior contained in $\mathcal{D}$. For otherwise, this would imply the existence of a surface in $M_{3}$, contained in the pre-image of $\Sigma$, with boundary $\tilde{a}_{i}-d_{2}(\tilde{a}_{i})$, which would contradict Lemma \ref{basis}. For the same reason, the interior of $\Sigma$ can not intersect $\partial \mathcal{D}$ along a set of curves that are null homologous in $M_{3}$, as in Figure \ref{boundaryint} (a).\\

\begin{figure}
\centering
\def\svgwidth{10cm}
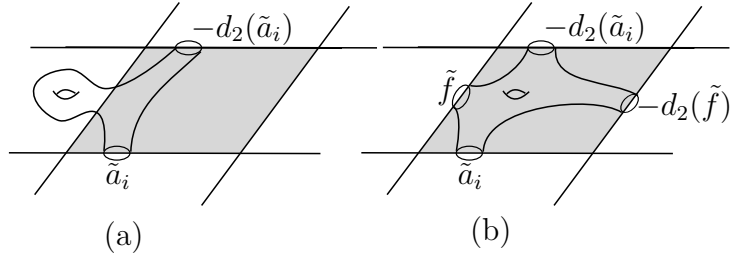
\caption{The shaded diamonds represent $\mathcal{D}$. The pre-image of $\Sigma$ can not intersect $\mathcal{D}$ as in part (a). A more realistic representation of a connected component of the pre-image of $\Sigma$ is shown in part (b).}
\label{boundaryint}
\end{figure}

\begin{figure}
\centering
\def\svgwidth{17cm}
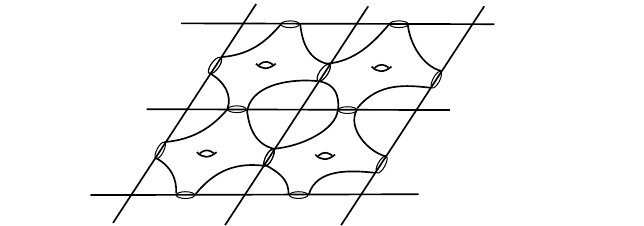
\caption{}
\label{union}
\end{figure}

Consider a connected component of the pre-image of $\Sigma$ in $M_{1}$. By Lemmas \ref{lemmachen} and \ref{andyslemma} part (4), the projection to $M_{1}$ of $\tilde{a}_{i}$ and $d_{2}(\tilde{a}_{i})$ are homologous, so it can be assumed without loss of generality that there is a union of connected components, $\Sigma_{1}$, of the pre-image of $\Sigma$ with boundary the projection to $M_{1}$ of $\tilde{a}_{i}-d_{2}(\tilde{a}_{i})$. Using the product structure $\times_{1}$, in $M_{1}\simeq S_{1}\times_{1}\mathbb{R}$ the surface $\Sigma_{1}$  can be homotoped so that its interior is disjoint from $S_{1}\times \{0\}$ and $S_{1}\times \{k\}$. The surface $\Sigma_{1}$ therefore intersects $\partial \mathcal{D}$ along a set of curves $f$ on $S_{2}$.\\

If $c$ is a curve in $M$, denote by $\tilde{c}$ a connected component of the pre-image of $c$ in $M_{3}$. When $c$ is a set of curves in $S_{2}$ disjoint from $S_{1}$, $\tilde{c}$ will denote the connected components of the pre-image of $c$ contained in $\mathcal{F}_{2}\subset \partial\mathcal{D}$. The orientation of the set of curves $\tilde{f}$ in $\mathcal{F}_{2}$ is chosen such that $\tilde{f}$ has a connected component of the lift of $\Sigma$ to its left.\\

It follows from Lemma \ref{basis} that 
\begin{equation}
\label{s}
d_{1}^{*}[\tilde{s}]=[\tilde{s}]=d_{2}^{*}[\tilde{s}]
\end{equation}
for any $\tilde{s}$ projecting onto the homology class $[s]$ in $M$. The lemma is therefore true for $[\tilde{s}]$. If $[\tilde{f}]=-[\tilde{a}_{i}]$ or $[\tilde{f}]=-[\tilde{a}_{i}]+\lambda[\tilde{s}]$ $\lambda\in \mathbb{Z}$, in $H_{1}(M_{3};\mathbb{Z})$, this implies the existence of a surface in $M_{3}$ with boundary curves homologous to $\{(d_{1}(\tilde{a}_{i})-\tilde{a}_{i})-(d_{2}(\tilde{a}_{i})-\tilde{a}_{i})\}$. This concludes the proof of Lemma \ref{eitheror} for $d^{-1}_{2}d_{1}$ using the homology class $[\tilde{a}_{i}]$. Likewise, if $[\tilde{f}]=[d_{2}(\tilde{a}_{i})]+\lambda[\tilde{s}]$, then Lemma \ref{eitheror} follows for $d_{2}d_{1}$ using the homology class $[\tilde{a}_{i}]$. Assume now that for all $\lambda\in \mathbb{Z}$, and all $\tilde{a}_{i}$, $$[\tilde{f}]\neq -[\tilde{a}_{i}]+\lambda[\tilde{s}]\qquad\text{ and }\qquad[\tilde{f}]\neq [\tilde{s}]\qquad\text{ and }\qquad[\tilde{f}]\neq [d_{2}(a_{i})]+\lambda[\tilde{s}].$$ In addition, it is easy to see that $$[\tilde{f}]\neq \lambda_{1}[\tilde{a}_{i}]+\lambda_{2}[\tilde{s}]\qquad\text{and}\qquad[\tilde{f}]\neq \lambda_{1}[d_{2}(a_{i})]+\lambda_{2}[\tilde{s}]\qquad\text{ for all }\lambda_{1}, \lambda_{2}\in \mathbb{Z}-\{\pm1\}.$$ This is because when $\lambda_{1}\neq \pm 1$, $\tilde{f}-d_{1}(\tilde{f})$ is not a primitive homology class in $H_{1}(M_{3};\mathbb{Z})$, and so by Lemma \ref{basis} could not be homologous to the primitive homology class with representative $d_{2}(\tilde{a}_{i})-\tilde{a}_{i}$.\\

Recall that algebraic intersection number in $M_{3}$ is defined by projecting curves onto $S_{3,1}$ or equivalently $S_{3,2}$. When $[\tilde{f}]$ is not in the span of $\{d_{1}^{*n}d_{2}^{*}[\tilde{a}_{i}]\text{, }d^{*n}_{1}[\tilde{a}_{i}] \text{, } [\tilde{s}]\mid n\in \mathbb{N}\}$ for fixed $i$, we can find a closed curve $\tilde{e}$ in $S_{3,2}$ with nonzero algebraic intersection number with $\tilde{f}$ but with zero algebraic intersection number with each curve in the set $\{d_{1}^{*n}d_{2}^{*}[\tilde{a}_{i}]\text{, }d^{*n}_{1}[\tilde{a}_{i}]\mid n\in \mathbb{N}\}$. Since  $\tilde{e}$ is a closed curve in $M_{3}$, and $e$ has finite geometric intersection number with $f$ in $M$, it is not possible that $\tilde{e}$ intersects every curve in the set $\{d_{1}^{n}(\tilde{f})\mid n\in \mathbb{N}\}$. Construct a surface $\mathcal{S}$ in $M_{3}$ by attaching a lift of $\Sigma_{1}$, call it $\tilde{\Sigma}_{1}$, to translates of $\tilde{\Sigma}_{1}$, as follows:
\begin{equation*}
\mathcal{S}=\tilde{\Sigma}_{1}\cup_{d_{1}(\tilde{f})}d_{1}(\tilde{\Sigma}_{1})\cup_{d_{1}^{2}(\tilde{f})}d_{1}^{2}(\tilde{\Sigma}_{1})\cup \ldots\cup_{d^{m}_{1}(\tilde{f})}d_{1}^{m}(\tilde{\Sigma}_{1})
\end{equation*}
If $m$ is chosen to be large enough that $\tilde{e}$ and $d^{m}_{1}(\tilde{f})$ are disjoint, the boundary of $\mathcal{S}$ has nonzero algebraic intersection number with $\tilde{e}$. This contradiction implies that $[\tilde{f}]$ must be in the span of $\{d_{1}^{*n}d_{2}^{*}[\tilde{a}_{1}]\text{, }d^{*n}_{1}[\tilde{a}_{1}]\text{, }[\tilde{s}]\mid n\in \mathbb{N}\}$. However, unless we have one of the special cases $[\tilde{f}]=\pm[\tilde{a}_{i}]+\lambda[\tilde{s}]$ or $[\tilde{f}]=\pm[d_{2}(\tilde{a}_{i})]+\lambda[\tilde{s}]$, Lemma \ref{basis} then implies that $d_{1}^{*}[\tilde{f}]$ can not be in the span of $\{d_{1}^{*n}d_{2}^{*}[\tilde{a}_{1}]\text{, }d^{*}_{1}[\tilde{a}_{1}], [\tilde{s}]\mid n\in \mathbb{Z}\setminus (\mathbb{N}\setminus \{1\})\}$. A contradiction can then be obtained similarly to the previous case, only now $\mathcal{S}$ is constructed out of a union of translations of $\tilde{\Sigma}_{1}$ under powers of $d_{1}^{-1}$ instead of $d_{1}$.\\

Lemma \ref{eitheror} for the homology class $[\tilde{a}_{i}]$ now follows from the following claim: $$[\tilde{f}]\neq[\tilde{a}_{i}]+\lambda[\tilde{s}]\qquad\text{ and }\qquad[\tilde{f}]\neq-[d_{2}(\tilde{a}_{i})]+\lambda[\tilde{s}].$$ 
Proof of claim: Suppose $[\tilde{f}]=[\tilde{a}_{i}]+\lambda[\tilde{s}]$. In the argument from the previous paragraph, replace $d_{2}$ by $d_{2}^{2}$, and $\mathcal{D}$ by $\mathcal{D}^{'}:=\mathcal{D}\cup_{d_{2}(\mathcal{F}_{1})}d_{2}(\mathcal{D})\cup d_{1}(\mathcal{D}\cup_{d_{2}(\mathcal{F}_{1})}d_{2}(\mathcal{D}))$, as in Figure \ref{union}. Then $\tilde{f}$ is replaced by $\tilde{f}\cup d_{2}(\tilde{f})$.\\

By Lemma \ref{andyslemma} part (4), the curves $d_{1}(\tilde{a}_{i})$ and $d_{2}(\tilde{a}_{i})$ project to homologous curves in $M$. It therefore follows from the argument just given with $\mathcal{D}^{'}$ in place of $\mathcal{D}$, that the homology class $[\tilde{f}]+[d_{2}(\tilde{f})]$ can only be in the span of $\{d_{1}^{*n}d_{2}^{*2}[\tilde{a}_{1}]\text{, }d^{*n}_{1}[\tilde{a}_{1}]\text{, }[\tilde{s}]\mid n\in \mathbb{N}\}$ if $[d_{1}(\tilde{a}_{i})]=[d_{2}(\tilde{a}_{i})]$. However, since $d_{1}^{*}[\tilde{s}]=d_{2}^{*}[\tilde{s}]=[\tilde{s}]$, the fact that the boundary of $\tilde{\Sigma}_{1}$ is null homologous then implies that $2[\tilde{a}_{i}]-[d_{2}(\tilde{a}_{i})]-[d_{1}(\tilde{a}_{i})]=2[\tilde{a}_{i}]-2[d_{1}(\tilde{a}_{i})]$, contradicting Lemma \ref{basis}. Hence $[\tilde{f}]\neq[\tilde{a}_{i}]+\lambda[\tilde{s}]$. The argument showing $[\tilde{f}]\neq-[d_{2}(\tilde{a}_{i})]+\lambda[\tilde{s}]$ is analogous.\qed\\

Going back to the proof of Lemma \ref{eitheror}, note that the above discussion applies to a curve $\tilde{b}_{j}$ on $\partial \mathcal{D}$ in the pre-image of $S_{1}$, for which $d_{2}(\tilde{b}_{i})$ is also on $\partial \mathcal{D}$. A basis in Lemma \ref{basis} can be chosen so that for any fixed $b_{j}$, this is achieved by homotoping the fiber $S_{2}$ around its base space in $M$ if necessary. Using this, Equation \ref{s} and Lemma \ref{basis}, it follows that there is a basis for $H_{1}(M_{3};\mathbb{Z})$, each element of which is preserved by at least one of $d_{2}^{-1}d_{1}$ and $d_{2}d_{1}$. \\

It remains to show that at least one of $d_{2}^{-1}d_{1}$ and $d_{2}d_{1}$ preserves every element of a basis of $H_{1}(M_{3};\mathbb{Z})$. Suppose this is not the case. By homotoping the fiber $S_{2}$ around its base space in $M$ if necessary, we can assume that there are curves $\tilde{c}_{1}$ and $\tilde{c}_{2}$ such that:
\begin{itemize}
\item $\tilde{c}_{1}$ and $\tilde{c}_{2}$ are both in $\mathcal{F}_{1}$.
\item $d_{2}^{*-1}d_{1}^{*}[\tilde{c}_{1}]=[\tilde{c}_{1}]$ but $d_{2}^{*-1}d_{1}^{*}[\tilde{c}_{2}]\neq[\tilde{c}_{2}]$. Then $d_{2}^{*}d_{1}^{*}[\tilde{c}_{2}]=[\tilde{c}_{2}]$
\item $[\tilde{c}_{1}]$ is not in the span of $[\tilde{c}_{2}]$ and $[\tilde{s}]$
\item $[\tilde{c}_{1}]\neq [\tilde{s}]$ and $[\tilde{c}_{2}]\neq [\tilde{s}]$
\end{itemize}
Let $\tilde{c}_{3}$ be a curve homologous to $[\tilde{c}_{1}]+[\tilde{c}_{2}]$ on the same connected component of the pre-image of $S_{1}$ on $\partial \mathcal{D}$ as $\tilde{c}_{1}$ and $\tilde{c}_{2}$. We have seen that either $d_{2}^{*-1}d_{1}^{*}[\tilde{c}_{3}]=[\tilde{c}_{3}]$ or $d_{2}^{*}d_{1}^{*}[\tilde{c}_{3}]=[\tilde{c}_{3}]$. In the first case, $d_{2}^{*-1}d_{1}^{*}[\tilde{c}_{3}]-[\tilde{c}_{3}]=0$ implies $d_{2}^{*-1}d_{1}^{*}([\tilde{c}_{1}]+[\tilde{c}_{2}])-([\tilde{c}_{1}]+[\tilde{c}_{2}])= d_{2}^{*-2}[\tilde{c}_{2}]-[\tilde{c}_{2}]=0$. This contradicts Lemma \ref{basis}. Similarly if $d_{2}^{*}d_{1}^{*}[\tilde{c}_{3}]=[\tilde{c}_{3}]$, a contradiction is reached.
\end{proof}



Note that Lemma \ref{eitheror} does not require the assumption that the monodromies $\tau_{1}$ and $\tau_{2}$ are pseudo-Anosov. In Example \ref{example}, Lemma \ref{eitheror} is immediate; $d_{2}^{-1}d_{1}$ does not change the homotopy class of curves in $M_{3}$.


\section{Proof of Theorem \ref{maintheorem}}
\label{lastone}

This section uses Lemma \ref{eitheror} to prove Theorem \ref{maintheorem} by contradiction. The next definition 
will be helpful in the process.\\

\begin{figure}
\centering
\def\svgwidth{10cm}
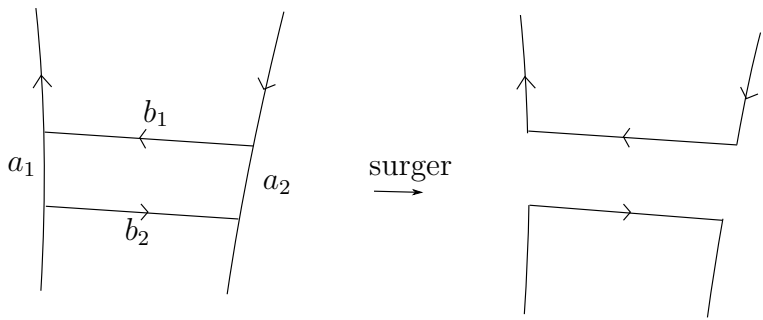
\caption{Surgering a curve $\alpha$ along an arc $b_{1}$.}
\label{1handle}
\end{figure}

\textbf{Surgeries.} Suppose $\alpha$ is an oriented set of curves on a surface, and $b_1$ is an oriented arc with endpoints on $\alpha$. Suppose also that extending $b_1$ out slightly past its endpoints would give a pair of crossings with $\alpha$ of opposite handedness. An arc with this property will be called an innermost arc. The arc $b_{1}$ determines an alteration to the set of curves $\alpha$ in $S$ that will be called surgering $\alpha$ along $b_{1}$. Let $b_2$ be a slight push off of $b_1$ keeping the endpoints on $\alpha$ so that $b_{1}$ and $b_2$ are disjoint. Then there exist small subarcs $a_1$ and $a_2$ of $\alpha$ joining the endpoints of $b_1$ and $b_2$, as shown in Figure \ref{1handle}. Surgering $\alpha$ along $b_1$ involves cutting the arcs $a_1$ and $a_2$ out of $\alpha$ and gluing in the arcs $b_1$ and $b_2$ in such a way that the orientations match up. A surgery does not change the homology class of $\alpha$.\\

We now have all the ingredients for the proof of Theorem \ref{maintheorem}.\\

\begin{proof}[Proof of Theorem \ref{maintheorem}]
In contradiction to Theorem \ref{maintheorem}, we assume that $S_{1}$ and $S_{2}$ are not isotopic. From Lemma \ref{eitheror}, one of $d^{-1}_{2}d_{1}$ or $d_{1}d_{2}$ acts trivially on $H_{1}(M_{3};\mathbb{Z})$. Therefore, we have two cases to consider; the first of which is easier, and does not require the assumption that the monodromies are pseudo-Anosov, and the second of which is more involved, and  does require the assumption that the monodromies are pseudo-Anosov.\\ 

\textbf{Case 1:} - $d_{1}d_{2}$ acts trivially on $H_{1}(M_{3};\mathbb{Z})$. Construct an embedded surface $S_{4}$ in $M_{3}$ as follows: Let $$\{s_{1}, s_{2}, \ldots, s_{m}\mid [s_{1}]+[s_{2}]+\ldots+[s_{m}]=[s]\}$$ be a set of curves in the intersection of $S_{3,1}$ and $S_{3,2}$. Take a fundamental domain of $S_{3,1}$ with boundary the set of curves $\{\tilde{s}_{i}, d_{1}(\tilde{s}_{i})\}$ and attach it to a fundamental domain of $S_{3,2}$ with boundary $\{d_{1}(\tilde{s}_{i}), d_{2}d_{1}(\tilde{s}_{i})\}$ along the common boundary components $\{d_{1}(\tilde{s}_{i})\}$. The surface $S_{4}$ is then obtained by taking the orbit under $d_{2}d_{1}$ of this union of fundamental domains, and attaching the connected components along common boundary curves to obtain a connected surface.\\

\begin{figure}
\centering
\def\svgwidth{8cm}
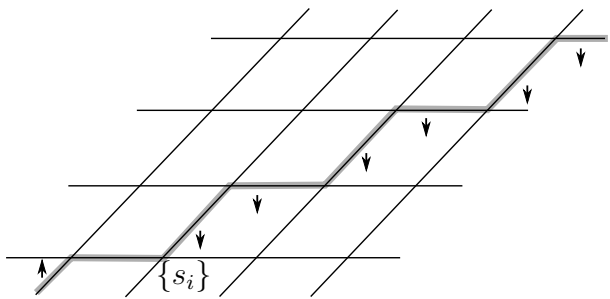
\caption{The surface $S_{4}$ is shown in grey, and the arrows represent the homotopy of $S_{4}$ onto $S_{3,1}$.}
\label{schematichomotopy}
\end{figure}

Recall that the diamonds each inherit a product structure from $\times_{1}$ and $\times_{2}$. Using the product structure $\times_{1}$, it can be seen that $S_{4}$ is homotopic to $S_{3,1}$ in $M_{3}$. This homotopy is illustrated schematically in Figure \ref{schematichomotopy}. However, if $d_{2}d_{1}$ acts trivially on $H_{1}(M_{3};\mathbb{Z})$, this implies that $H_{1}(S_{3};\mathbb{Z})$ has finite rank, contradicting Lemma \ref{basis}. \qed\\

\textbf{Case 2} - $d^{-1}_{2}d_{1}$ acts trivially on $H_{1}(M_{3};\mathbb{Z})$. A surface $S_{4}$ in $M_{3}$ is now constructed to be an infinite cyclic cover of a surface in $M$, with deck transformation group generated by $d_{2}^{-1}d_{1}$. This may be done for example by attaching two fundamental domains, $\mathcal{F}_{1}$ and $\mathcal{F}_{2}$ along common boundary curves.  Here $\mathcal{F}_{1}$ is a fundamental domain of $S_{3,1}\rightarrow S_{1}$ with boundary the set of curves $\{\tilde{s}_{i}, d_{1}(\tilde{s}_{i}\}$, and $\mathcal{F}_{2}$ is a fundamental domain of $S_{3,2}\rightarrow S_{2}$ with boundary $\{d_{1}(\tilde{s}_{i}), d_{2}^{-1}d_{1}(\tilde{s}_{i})\}$ and orientation opposite to that of a subsurface of $S_{3,2}$. The set of curves $\{s_{i}\}$ are in the intersection of $S_{3,1}$ and $S_{3,2}$. The details of the construction of $S_{4}$ are not important; we will only require it to be embedded and oriented.\\

\begin{figure}
\centering
\def\svgwidth{8cm}
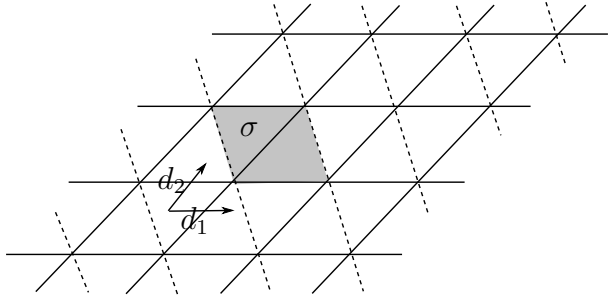
\caption{A fundamental domain $\sigma$ is shaded. The orbit of the surface $S_{4}$ under the action of the deck transformation group is represented by a dotted line.}
\label{toiletroll}
\end{figure}

A fundamental domain $\sigma$ for the cover $M_{3}\rightarrow M$ is illustrated schematically in Figure \ref{toiletroll}. The boundary of $\sigma$ consists of a subsurface of $S_{4}$, a subsurface of $d_{2}(S_{4})$, a subsurface of $S_{3,2}$, and the image of this subsurface under $d_{2}^{-1}d_{1}$. Let 
\begin{equation*}
\sigma_{\infty}:=\cup_{n=-\infty}^{n=\infty}(d_{2}^{-1}d_{1})^{n}(\sigma)
\end{equation*}
Using Lemma \ref{eitheror} we see tha the image of $H_{1}(\sigma_{\infty};\mathbb{Z})$ in $H_{1}(M_{3};\mathbb{Z})$ has finite rank. Identifying points on $\partial \sigma_{\infty}$ under the action of $d_{2}$ gives the manifold $M_{2}\simeq S_{2}\times_{2}\mathbb{R}$. Let $c$ be a simple, nonseparating curve on $S_{2}\times \{0\}\subset M_{2}$, $[c]\neq [s]$, and such that $c$ is disjoint from the projection $p(\partial \sigma_{\infty})$ of $\partial \sigma_{\infty}$ to $M_{2}$. The contradiction is found by showing that curves in $\sigma_{\infty}$ projecting onto curves in $M_{2}$ in the homology class $[c]$ can not be contained in a submodule of $H_{1}(S_{3};\mathbb{Z})$ of finite rank.\\

If $g$ is a curve in $M_{2}\setminus p(\partial \sigma_{\infty})$, denote by $\tilde{g}$ its pre-image in $\sigma_{\infty}$. In this context, $g$ is assumed to be a specific representative of a free homotopy class of curves in $M_{2}$, so its pre-image consists of a single curve in $\sigma_{\infty}$.\\

Since $p(\partial \sigma_{\infty})$ is oriented, in a sufficiently small $\epsilon$-neighbourbourhood of $p(\partial \sigma_{\infty})$ in $M_{2}$, it is possible to categorise points as being ``to the left'', ``to the right'' or ``on'' $p(\partial \sigma_{\infty})$. This is done in such a way that homotoping the curve $g$ in $M_{2}\setminus p(\partial \sigma_{\infty})$ over $p(\partial \sigma_{\infty})$ from left to right gives a curve $g^{'}$ with the property that $\tilde{g}^{'}$ is homotopic to $d_{2}(\tilde{g})$. \\

We now need to use the crucial assumption that the monodromies are pseudo-Anosov. By construction, $p(\partial \sigma_{\infty})\cap (S_{2}\times\{nk\})$ in $M_{2}$ is equal to $\delta_{2}^{nk}(p(\partial \sigma_{\infty})\cap (S_{2}\times\{0\}))$. Since $\tau_{2}$ is pseudo-Anosov, as $n$ approaches infinity, the geometric intersection number of $p(\partial \sigma_{\infty})\cap (S_{2}\times\{nk\})$ with $(c\times \mathbb{R})\cap (S_{2}\times\{nk\})$ approaches infinity. Suppose a homotopy shifts the curve $c$ in the positive $\mathbb{R}$ direction. By an appropriate choice of orientations on the fibers in $M$, it can be assumed without loss of generality that there are infinitely many places at which an arc of $c$ is homotoped over $p(\partial \sigma_{\infty})$ from left to right. \\

As $c$ is homotoped in the positive $\mathbb{R}$ direction, if a point on the curve approaches $p(\partial \sigma_{\infty})$ from the right, further homotopes are assumed to fix this point. Keep homotoping the curve until an arc on the curve is moved over $p(\partial \sigma_{\infty})$.\\

Since $p(\sigma_{\infty})$ is connected, by surgering the curve homotopic to $c$ along an arc or arcs in $p(\sigma_{\infty})$, a set of curves $\{\beta_{1}, \beta_{2}\}$ is obtained, where $\beta_{1}$ and $\beta_{2}$ are each disjoint from $p(\sigma_{\infty})$. This decomposition is not unique, and can be done such that neither $[\beta_{1}]$ nor $[\beta_{2}]$ are $[0]$ or $[s]$. Choose a basepoint in $M_{2}\setminus p(\partial \sigma_{\infty})$, and assume all curves have been homotoped within $M_{2}\setminus p(\partial \sigma_{\infty})$ to pass through the basepoint. The subscripts ``1'' and ``2'' are assigned such that the pre-image of $\beta_{1}\cup \beta_{2}$ in $\sigma_{\infty}\subset M_{3}$ is homologous in $M_{3}$ to the pre-image of $\beta_{1}\circ s_{2}^{*}\circ \beta_{2}\circ s_{2}^{*-1}$ in $\sigma_{\infty}\subset M_{3}$. By Lemma \ref{basis}, the pre-image of $\beta_{1}\circ s_{2}^{*}\circ \beta_{2}\circ s_{2}^{*-1}$ in $\sigma_{\infty}$ is not homologous in $M_{3}$ to the pre-image of $c$ in $\sigma_{\infty}$.\\

Now repeat this construction on $\beta_{1}$ and $\beta_{2}$. This time, $\beta_{1}$ is homologous in $M_{2}$ to $\gamma_{1}\cup\gamma_{2}$, and $\beta_{2}$ is homologous in $M_{2}$ to $\gamma_{3}\cup\gamma_{4}$. The pre-image of $\gamma_{1}\cup\gamma_{2}\cup\gamma_{3}\cup\gamma_{4}$ is therefore homologous in $M_{3}$ to the pre-image of $\gamma_{1}\circ s_{2}^{*}\circ \gamma_{2}\gamma_{3}\circ s_{2}^{*-1}\circ s_{2}^{*2}\circ \gamma_{4}\circ s_{2}^{*-2}$. By Lemma \ref{basis}, the homology class of the pre-image of this curve is not in the span of $[p^{-1}(c)]$ and $[p^{-1}(\beta_{1}\circ s_{2}^{*}\circ \beta_{2}\circ s_{2}^{*-1})]$. Also by Lemma \ref{basis}, the homology classes in $M_{3}$ of the pre-images of $c$, $\beta_{1}\circ s_{2}^{*} \circ \beta_{2} \circ s_{2}^{*-1}$ and $\gamma_{1}\circ s_{2}^{*}\circ \gamma_{2}\gamma_{3}\circ s_{2}^{*-1}\circ s_{2}^{*2}\circ \gamma_{4}\circ s_{2}^{*-2}$ can not be contained in a submodule of $H_{1}(M_{3};\mathbb{Z})$ of rank less than three. This completes the proof of case 2 and of the theorem.
\end{proof}

\begin{rem}\label{rem.nopA}
Note that when $\tau_{1}$ and $\tau_{2}$ are not pseudo-Anosov, $p(\partial \sigma_{\infty})$ could be a union of annuli, and so there is no possibility of deriving the promised contradiction. 
\end{rem}

\bibliographystyle{plain}
\bibliography{torellifacebib}

\end{document}